\documentclass[12pt]{amsproc}
\usepackage{amsmath,amsthm}
\usepackage{amsfonts,amssymb}
\newtheorem{thm}{Theorem}[section]
\newtheorem{cor}[thm]{Corollary}
\newtheorem{lem}[thm]{Lemma}

\theoremstyle{remark}
\newtheorem{rem}{Remark}[section]
\newtheorem{exam}{Example}[section]

\def\CB{{\mathcal B}}
\def\CF{{\mathcal F}}

\def\CH{{\mathcal H}}

\def\CO{{\mathcal O}}

\def\CS{{\mathcal S}}

\def\CP{{\mathcal P}}

\def\a{{\mathfrak a}}

\def\u{{\mathfrak u}}
\def\k{{\mathfrak k}}
\def\p{{\mathfrak p}}

\def\C{{\mathbb C}}

\def\N{{\mathbb N}}
\def\R{{\mathbb R}}

\def\1{\text{\bf {1}}}

\begin{document}

\title[Toeplitz operators]
{Toeplitz operators with special symbols\\
on Segal-Bargmann spaces}
\author{Jotsaroop K. and S. Thangavelu}

\address{Department of Mathematics\\ Indian Institute
of Science\\Bangalore-560 012}
\email{jyoti@math.iisc.ernet.in, veluma@math.iisc.ernet.in}

\date{\today}
\thanks{}

\begin{abstract}
We study the boundedness of Toeplitz operators on Segal-Bargmann spaces in
various contexts. Using Gutzmer's formula as the main tool we identify symbols
for which the Toeplitz operators correspond to Fourier multipliers on the
underlying groups. The spaces considered include Fock spaces, Hermite and
twisted Bergman spaces and Segal-Bargmann spaces associated to Riemannian
symmetric spaces of compact type.
\end{abstract}

\maketitle

\section{Introduction}
\setcounter{equation}{0}

Given a domain $ \Omega $ in $ \C^n $ let $ \CH(\Omega,d\mu) $ stand for a
weighted Bergman space of holomorphic functions contained in
$ L^2(\Omega,d\mu).$ Let $ g $ be 
a Lebesgue measurable function on $ \Omega $ such that 
$ gF \in L^2(\Omega, d\mu)$ for all
$ F  $ from a dense subspace of $  \CH(\Omega,d\mu).$ We can then define 
the Toeplitz operator $ T_g $ by
$ T_g F = P(gF) $ where $ P: L^2(\Omega,d\mu) \rightarrow \CH(\Omega,d\mu) $
is the natural orthogonal projection. Such Toeplitz operators have been
studied extensively in the literature.

Suppose now that $ \Omega $ is invariant under the action of a
Lie group $ G .$ The group $ G $ has a natural action on $ \CH(\Omega,d\mu) .$ Let us further assume that there is an isometric isomorphism $ B $ between
$ L^2(G) $ and $ \CH(\Omega,d\mu).$ Using this, we can transfer the Toeplitz
operator $ T_g $ into the operator $ B^{-1}T_gB $ acting on $ L^2(G).$ Then
the boundedness of $ T_g $ becomes equivalent to that of this transferred
operator which might turn out to be easier to study using harmonic analysis on
the group $ G.$ The simplest bounded operators on $ L^2(G) $ are given by
Fourier multilpiers and hence it is natural to ask which Toeplitz operators
give rise to such multiplier transformations.

In this article we are interested in the case where $ B $ is the
Segal-Bargmann transform on the group $ G $ and $\CH(\Omega,d\mu) $ is the
image of $ L^2(G) $ under $ B.$ It turns out that we can identify a large
class of symbols $ g $ for which $  B^{-1}T_gB $ reduces to Fourier
multipliers. The groups for which such results can be proved include $ \R^n $,
Heisenberg groups and compact Lie groups. An important role is played by the
so called Gutzmer's formula. Thus for Fock spaces, Hermite and twisted Bergman spaces and Segal-Bargmann spaces associated to compact Lie groups and symmetric spaces we have identified special classes of symbols for which $ T_g$'s
correspond to multiplier transforms.

The plan of the paper is as follows. In the next section we look at Toeplitz 
operators on the classical Fock spaces. In section 3 we study Toeplitz 
operators on Hermite-Bergman spaces which give rise to Hermite multipliers 
when conjugated with the Hermite semigroup. In section 4 we characterise all 
Toeplitz operators on the twisted Bergman spaces that correspond to Weyl 
multipliers. Finally, in the last section we consider Toeplitz operators on 
Segal-Bargmann spaces associated to compact Lie groups and symmetric spaces. 
For results closely  related to the theme of this paper we refer to \cite{BC}, 
\cite{H1}, \cite{HL} and \cite{H2}.

\section{Toeplitz operators on Fock spaces}
\setcounter{equation}{0}

In this section we look at Toeplitz operators on Fock spaces which have been
studied by several authors, see \cite{BC} and the references there. First we
consider Toeplitz operators with radial symbols and obtain a necessary and
sufficient condition for $ T_g $ to be bounded. Toeplitz operators with
radial symbols on $\mathcal{F}(\mathbb{C}) $ with a different assumption have
been studied by Grudsky and Vasilevski \cite{Gr}. The condition involves the
heat flow $ g*q_{1/4} $ and under a mild decay assumption we prove boundedness
when $ g $ is radial. Later we consider symbols $ g(x+iy) $ which depend only
on $ y $ and show that they correspond to Fourier multipliers. For such
symbols we show that the conjecture of Berger and Coburn \cite{BC} is true.

\subsection{ Radial symbols}

In this subsection we consider Toeplitz operators associated to radial symbols on the Fock space $\mathcal{F}(\mathbb{C}^n)$.
$$\mathcal{F}(\mathbb{C}^n):= \{f\in {\mathcal{O}(\mathbb{C}^n)} : \int_{\mathbb{C}^n}|f(z)|^2 d\mu(z)<\infty\},$$
where $d\mu(z)= (2\pi)^{-n}e^{-\frac{1}{2}|z|^2}$. It is known that $\CF(\C^n)$ is a Hilbert space with the reproducing kernel explicitly given by $K(z,w)= e^{\frac{z\cdot\bar{w}}{2}}$.
 Recall that the Toeplitz operator with symbol $ g $ is given by
$$ T_gf(z) = P(fg)(z) = \int_{\C^n} g(w)f(w)e^{z\cdot \bar{w}/2} d\mu(w).$$
An orthonormal basis for the Fock space   $\mathcal{F}(\mathbb{C}^n)$ is
given by
$$ \zeta_\alpha(z) = \frac{z^\alpha}{2^{\frac{|\alpha|}{2}}(\alpha !)^{1/2}}.$$
As $ \langle T_g \zeta_\alpha,\zeta_\beta\rangle = \langle g \zeta_\alpha,\zeta_\beta\rangle $ we can easily check that  $$ \langle T_g \zeta_\alpha,\zeta_\beta\rangle = \delta_{\alpha \beta}
 \langle T_g \zeta_\alpha,\zeta_\alpha\rangle $$ whenever $ g $ is radial. This leads to
the following result for Toeplitz operators with radial symbols. In what
follows we let  $$ q_t(z) = (4\pi t)^{-n} e^{-\frac{1}{4t}|z|^2}$$
stand for  the heat kernel on $ \C^n $ associated to the standard Laplacian
and
$$ \varphi_k(z) = L_k^{n-1}(\frac{1}{2}|z|^2)e^{-\frac{1}{4}|z|^2}$$
for  the Laguerre functions of type $ (n-1).$

\begin{thm} Let $ g $ be a measurable function on $ \C^n $ such that $ g
\zeta_\alpha  \in L^2(\C^n, d\mu) $ for all $ \alpha \in \N^n.$ Then the  Toeplitz
operator $ T_g $ is bounded on $ \CF(\C^n) $ if and only if  the sequence
$$ \frac{k!(n-1)!}{(k+n-1)!} \int_{\C^n} g*q_{1/4}(w)\varphi_k(2w) dw $$
is bounded.
\end{thm}
\begin{proof} Using the result (see Lemma 3.2.6 in \cite{Th6})
$$ \int_{S^{2n-1}} \zeta_\alpha(\omega)\overline{\zeta_\beta(\omega)}
d\sigma(\omega) = \delta_{\alpha \beta} \frac{(n-1)!}{(|\alpha|+n-1)!}
2^{-|\alpha|} $$ we easily calculate that whenever $ g $ is radial
$$ \langle T_g\zeta_\alpha,\zeta_\beta\rangle = \delta_{\alpha \beta}
\frac{k!(n-1)!}{(k+n-1)!} \int_0^\infty g(r) \frac{r^{2k+2n-1}}{2^kk!}
e^{-\frac{1}{2}r^2}dr $$ where $ k = |\alpha|.$ Therefore,
$$ T_gF(z) = \sum_{\alpha \in \N^n} R_{|\alpha|}(g)\langle F,\zeta_\alpha\rangle
\zeta_\alpha(z) $$
where
$$ R_k(g) = \frac{k!(n-1)!}{(k+n-1)!} \int_0^\infty g(r) \frac{r^{2k+2n-1}}{2^kk!}
e^{-\frac{1}{2}r^2} dr.$$ The theorem now follows from the following lemma.
\end{proof}

\begin{lem} Let $ g $ be a radial function as in the theorem. Then for any
$ k \in \N,$
$$ R_k(g) = c_n \frac{k!(n-1)!}{(k+n-1)!} \int_{\C^n}
g*q_{1/4}(w)\varphi_k(2w) dw .$$
\end{lem}
\begin{proof} We make use of the following formula satisfied by Laguerre
functions (see Szego \cite{Sz})
$$ e^{-x^2}L_k^\alpha(x^2) = \frac{1}{k!} \int_0^\infty e^{-t^2}t^{2k+\alpha}
\frac{J_\alpha(2tx)}{t^\alpha x^\alpha} t^\alpha dt $$
which can be rewritten as
$$ e^{-2x^2}L_k^\alpha(2x^2) = \frac{1}{k!} \int_0^\infty
e^{-\frac{1}{8}t^2}(\frac{t^2}{8})^k
\frac{J_\alpha(tx)}{t^\alpha x^\alpha} t^{2\alpha+1} dt .$$
Inverting the Hankel transform and making a change of variables we get
\begin{equation}\label{25}
 \frac{1}{k!}e^{-\frac{1}{2}t^2}(\frac{t^2}{2})^k = \int_0^\infty
e^{-2x^2}L_k^\alpha(2x^2)\frac{J_\alpha(2tx)}{(2tx)^\alpha} x^{2\alpha+1} dx.
\end{equation}
As both sides are holomorphic in $ t $ the above equation remains true when
$ t $ is replaced by $ it.$

Under the assumption that $ g $ is radial we observe that
$$ g*q_{1/4}(z) = \int_{\C^n} g(w) e^{-|z-w|^2} dw $$ reduces to a constant
multiple of
$$ e^{-|z|^2} \int_0^\infty g(r)e^{-r^2} \frac{J_{n-1}(2irs)}{(2irs)^{n-1}}
r^{2n-1} dr.$$ Therefore,
$$ \int_{\C^n} g*q_{1/4}(w)\varphi_k(2w) dw $$
$$ = c_n \int_0^\infty  \int_0^\infty e^{-r^2}g(r) e^{-s^2}
\frac{J_{n-1}(2irs)}{(2irs)^{n-1}}L_k^{n-1}(2s^2)
e^{-s^2} r^{2n-1}s^{2n-1} dr ds .$$
Using Fubini, which is justified by our assumptions on $ g $, and making use
of the above identity (\ref{25}) satisfied by Laguerre functions, we obtain
$$ \int_{\C^n} g*q_{1/4}(w)\varphi_k(2w) dw =
c_n  \int_0^\infty g(r) \frac{r^{2k+2n-1}}{2^k k!} e^{-\frac{1}{2}r^2}.$$
This completes the proof of the lemma.
\end{proof}

\begin{cor} Let $ g $ be a radial function as in the previous theorem.
Further assume that  $ |g*q_{1/4}(z)| \leq C|z|^{-1}$, for all $z\neq0.$ Then $ T_g $ is
bounded on $ \CF(\C^n).$
\end{cor}
\begin{proof} In view of the theorem we only need to check that the sequence
$$ \frac{k!(n-1)!}{(k+n-1)!} \int_{0}^\infty h(r)
 L_k^{n-1}(2r^2)e^{-r^2} r^{2n-1} dr $$
is bounded where $ h(z) = g*q_{1/4}(z).$ Under the assumption on $ g*q_{1/4}
$ this can be easily verified using the following estimates on  integrals
of Laguerre functions.
$$\left|\frac{k!(n-1)!}{(k+n-1)!} \int_{0}^\infty h(r)
 L_k^{n-1}(2r^2)e^{-r^2} r^{2n-1} dr \right|$$
 $$\leq c_n\frac{k!(n-1)!}{(k+n-1)!}\int_0^{\infty}r^{-1}|L_k^{n-1}(r^2)|e^{-r^2/2} r^{2n-1} dr .$$
 We define $\mathcal{L}_k^{n-1}(r^2), r\in\mathbb{R}$ by
 $$\mathcal{L}_k^{n-1}(r^2)=\left(\frac{k!(n-1)!}{(k+n-1)!}\right)^{1/2}L_k^{n-1}(r^2)r^{n-1}e^{-r^2/2}.$$
It follows from Lemma 1.5.4 in \cite{Th6} that
$$\int_0^{\infty}|\mathcal{L}_k^{n-1}(r^2)|r^{-\beta}r dr\sim k^{1/2-\beta/2}$$ when $k$ is large. By Stirling's formula for large $k$,  $\frac{k!(n-1)!}{(k+n-1)!}\sim k^{-(n-1)}.$
By using the estimates above after putting $\beta=-(n-2)$ we have
 $$\left|\frac{k!(n-1)!}{(k+n-1)!} \int_{0}^\infty h(r)
 L_k^{n-1}(2r^2)e^{-r^2} r^{2n-1} dr \right|$$
 $$\leq\left(\frac{k!(n-1)!}{(k+n-1)!}\right)^{1/2}\int_0^{\infty}|\mathcal{L}_k^{n-1}(r^2)|r^{n-2}r dr$$
 $$\sim k^{-(n-1)/2}k^{1/2+(n-2)/2} =1.$$
 This proves the lemma.
\end{proof}

\subsection{ Toeplitz operators and Fourier multipliers}

A conjecture of Berger and Coburn \cite{BC} says  that $ T_g $ is bounded on
$ \CF(\C^n) $ if and only if $ g*q_{1/4} $ is bounded. In this subsection we
verify this conjecture when the symbol $ g(x+iy) $ depends only on $ y.$ In such a case the problem reduces to checking if a certain Fourier multiplier is
bounded on $ L^2(\R^n).$ As the Fock space is closely related to the weighted
Bergman space associated to the Segal-Bargmann/heat kernel transform we
consider Toeplitz operators on the space $ \CB_t(\C^n) $ consisting of entire
functions that are square integrable with respect to $ q_{t/2}(y)dx dy $ where
$ q_t $ is the standard heat kernel on $ \R^n.$ By the results of Segal and
Bargmann \cite{B} we know that $ F \in \CB_t(\C^n)$ if and only if $ F = f*q_t $
for some $ f \in L^2(\R^n) $ and
$$ \int_{\R^{2n}} |F(x+iy)|^2 q_{t/2}(y)dx dy = c_n \int_{\R^n} |f(x)|^2 dx.$$
Let $ g $ be a measurable function on $ \C^n $ such that $ gF $
belongs to\newline
 $L^2(\C^n,q_{t/2}(y)dz) $ whenever $ F \in  \CB_t(\C^n) $ and let $ T_g $ be
the associated Toeplitz operator.

\begin{thm} Let $ g(x+iy) = g_0(y) $ be as above. Then $ T_g $ is
bounded on $ \CB_t(\C^n) $ if and only if $ g_0*q_{t/2} $ is bounded where the
convolution is on $ \R^n.$
\end{thm}
\begin{proof} When $ F_j = f_j*q_t \in \CB_t(\C^n), j = 1,2 $ Plancherel's
theorem leads to
$$ \int_{\R^n} F_1(x+iy) \overline{F_2(x+iy)} dx = c_n \int_{\R^n}
\hat{f_1}(\xi)\overline{\hat{f_2}(\xi)} e^{-2t|\xi|^2}
e^{-2 y\cdot \xi} d\xi.$$
Integrating the above with respect to $ g(y)q_{t/2}(y) dy $ we see that
$$\int_{\C^n} T_gF_1(x+iy) \overline{F_2(x+iy)} q_{t/2}(y)dx dy =
c_n \int_{\R^n} m_t(\xi) \hat{f_1}(\xi)\overline{\hat{f_2}(\xi)} d\xi $$
where
$$ m_t(\xi) = e^{-2t|\xi|^2}\int_{\R^n}e^{-2 y\cdot \xi} 
g_0(y)q_{t/2}(y) dy. $$
From this it is clear that $ T_g $ is bounded if and only if $ m_t $ defines
a bounded Fourier multiplier on $ L^2(\R^n) $ which happens precisely when
$ m_t $ is a bounded function. An easy calculation shows that  $ m_t(\xi) =
g_0*q_{t/2}(\xi) $ which proves the theorem.
\end{proof}

\begin{rem} We can read out properties of Fourier multipliers $ m_t(\xi) $
that correspond to Toeplitz operators from the work of Hille \cite{Hi}. Indeed, when
$ t = 1/2 $ which corresponds to the Fock space, the multilpier $ m $ and
the symbol $ g $ are related via
$$ m(\xi) = (2\pi)^{-n/2} \int_{\R^n} g_0(y) e^{-|\xi-y|^2} dy .$$ Assuming
$ n =1 ,$ let
$$ \int_{-\infty}^\infty |g_0(y)| e^{-y^2+\alpha |y|} dy < \infty $$
for some $ \alpha > 0. $ Then if $ g_0(y) = \sum_{k=0}^\infty a_k H_k(y) $ is
the expansion of $ g_0 $ in terms of the Hermite polynomials $ H_k ,$ Hille \cite{Hi}
has proved  that $ m(z) = \sum_{k=0}^\infty a_k (2 z)^k $ for all $ z \in \C $
with $ |z| < \alpha.$
\end{rem}

\section{Toeplitz operators on Hermite-Bergman spaces}
\setcounter{equation}{0}

In this section we study Toeplitz operators on Hermite-Bergman spaces which
are Segal-Bargmann spaces associated to the Hermite semigroup $ e^{-tH}$.
As in the case
of Fock spaces we show that the transferred operator $ e^{tH}T_ge^{-tH} $ is a
pseudo-differential operator whose Weyl symbol is related to the heat flow of
$ g $. This leads to a conjecture similar to that of Berger and Coburn. By
making use of Gutzmer's formula for Hermite expansions we identify certain
special symbols $ g $ which lead to Hermite multipliers.

\subsection{Hermite-Bergman spaces} On $ \R^{2n} $ consider the weight
function $ U_t $ given by
$$ U_t(x,y) = 4^n (\sinh(4t))^{-n/2}e^{\tanh(2t)|x|^2-\coth(2t)|y|^2}.$$ The
Hermite Bergman space $ \CH_t(\C^n) $ is the space of all entire functions
$ F $ which are square integrable with respect to $ U_t(x,y) dx dy.$ It is
known
that $ F \in \CH_t(\C^n) $ if and only if $ F = e^{-tH}f $ for some $ f \in
L^2(\R^n) $ where $ e^{-tH} $ is the Hermite semigroup, see \cite{By}. Moreover,
$$  \int_{\R^{2n}} |F(x+iy)|^2 U_t(x,y) dx dy = c_n \int_{\R^n} |f(x)|^2 dx $$
whenever $ F = e^{-tH}f.$ In the above the Hermite semigroup is defined by
$$ e^{-tH}f = \sum_{\alpha \in \N^n} e^{-(2|\alpha|+n)t}\langle f,\Phi_{\alpha}\rangle
\Phi_\alpha $$ where $ \Phi_\alpha $ are the normalised Hermite functions
which are eigenfunctions of the Hermite operator $ H  = -\Delta+|x|^2 $ with
eigenvalues $ (2|\alpha|+n).$ See \cite{Th5} for more about Hermite functions.

An important tool in studying the above space is an analogue of Gutzmer's
formula for Hermite expansions which we  now proceed to state. Let
$ \pi(x,u) $ be the family of unitary operators defined on $ L^2(\R^n) $ by
$$ \pi(x,u)\varphi(\xi) = e^{i(x \cdot \xi+\frac{1}{2}x \cdot y)}
\varphi(\xi+y).$$
These are related to the Schr\"odinger representation of the Heisenberg group,
see \cite{Th1} and \cite{Fo}. It is clear $ \pi(z,w)F(\xi) $ makes sense even
for $ (z,w) \in \C^n \times \C^n $ whenever $ F $ is holomorphic. However,
the resulting function need not be in $ L^2(\R^n) $ unless further
assumptions are made
on $ F.$ When $ F = \Phi_\alpha $ (or any finite linear combination of the
Hermite functions) $  \pi(z,w)F(\xi) $ is indeed in $ L^2(\R^n) $ and using
Mehler's formula for the Hermite functions we can prove that
$$  \int_{\R^n} |\pi(z,w)\Phi_\alpha(\xi)|^2 d\xi = (2\pi)^{\frac{n}{2}}
e^{(u\cdot y -v\cdot x)} \Phi_{\alpha,\alpha}(2iy,2iv) $$ where
$ \Phi_{\alpha,\alpha} $ are the special Hermite functions which are expressible
in terms of Laguerre functions. Gutzmer's formula says that a similar result
is true for $  \pi(z,w)F(\xi) $ under some assumptions on $ F.$

In order to state Gutzmer's formula we need to introduce one more notation.
Let $ Sp(n,\R) $ stand for the symplectic group consisting of
$ 2n \times 2n $ real matrices that preserve the symplectic form
$ [(x,u),(y,v)] = (u\cdot y-v\cdot x) $ on $ \R^{2n} $ and have determinant
one. Let $ O(2n,\R) $ be the orthogonal
group and we define $ K = Sp(n,\R)\cap O(2n,\R).$ Then there is a one to one
correspondence between $ K $ and the unitary group $ U(n) .$
Let $ \sigma = a+ib $ be an $ n \times n $ complex matrix with real and
imaginary parts $ a $ and $ b.$ Then $ \sigma $ is unitary if and only if
the matrix $ A = \begin{pmatrix}a & -b \cr b & a \end{pmatrix}$ is
in $ K.$ For these facts we refer to Folland \cite{Fo}. By $ \sigma.(x,u) $ we
denote the action of the correspoding matrix $ A $ on $ (x,u).$ This action
has a natural extension to $ \C^n \times \C^n $ denoted by $ \sigma.(z,w) $
and is given by $ \sigma.(z,w) = (a.z-b.w, a.w+b.z) $ where $ \sigma = a+ib.$

\begin{thm} For a holomorphic function  $ F $  we have the following formula
for any  $ z = x+iy, w= u+iv \in \C^n $:
$$ \int_{\R^n} \int_K |\pi(\sigma.(z,w))F(\xi)|^2 d\sigma d\xi $$
$$ =  e^{(u\cdot y - v \cdot x)} \sum_{k=0}^\infty \frac{k!(n-1)!}{(k+n-1)!}
\varphi_k(2iy,2iv) \|P_kf\|_2^2 $$
where $ f $ stands for  the restriction of $ F $ to $ \R^n.$
\end{thm}

In the above formula $ P_k $ are the spectral projections of the Hermite
operator defined by
$$ P_kf(x) = \sum_{|\alpha|=k} \langle f,\Phi_\alpha\rangle\Phi_\alpha(x) $$ and
$$ \varphi_k(z,w) = L_k^{n-1}(\frac{1}{2}(z^2+w^2))e^{-\frac{1}{4}(z^2+w^2)}
$$ are the holomorphically extended Laguerre functions of type
$ (n-1).$ The above formula means that if either the integral or the sum is
finite then they are equal. Note that the sum is clearly finite when
$ f = e^{-tH}g $ for some $ g \in L^2(\R^n).$ We refer to \cite{Th4} for a
proof of the above formula. The characterisation of $ \CH_t(\C^n) $ as the
image of $ L^2(\R^n) $ under the Hermite semigroup $ e^{-tH} $ can be proved
using Gutzmer's formula, see \cite{Th4}. The only other ingredient needed is the
formula
$$ \frac{k!(n-1)!}{(k+n-1)!}\int_{\R^{2n}} p_{2t}(2y,2v)\varphi_k(2iy,2iv)
dy dv = e^{2(2k+n)t} $$
where $ p_t(y,v) $ stands for the heat kernel associated to the special
Hermite operator, see Section 4.

\subsection{ Toeplitz operators on $ \CH_t(\C^n)$}

Let $ P : L^2(\C^n) \rightarrow \CH_t(\C^n) $ be the orthogonal projection
which is explicitly given by
$$ PF(z) = \int_{\R^{2n}} F(u,v) K_t(z,u+iv) U_t(u,v) du dv .$$
Here $ K_t(z,w) $ is the reproducing kernel of $ \CH_t(\C^n) $ defined by
$$ K_t(z,w) = \sum_{\alpha \in \N^n} e^{-2(2|\alpha|+n)t}\Phi_\alpha(z)
\Phi_\alpha(\bar{w}).$$ Using Mehler's formula we can show that
$$ K_t(z,w) = (\sinh(4t))^{-\frac{n}{2}} e^{-\frac{1}{2}\coth(4t) (z^2+w^2)}
e^{\frac{1}{\sinh(4t)}\langle z,w\rangle},$$
where $\langle z,w\rangle$ is the standard Hermitian inner product on
$\C^n $ and $ z^2 = z_1^2+...+z_n^2 $ etc.
For a  measurable function $g$ on $\mathbb{C}^n$ such that $gK_t(.,w)$ belongs to $L^2(\mathbb{C}^n, d\mu_t)$ for all $w$ (we will refer to this condition as $\ast$), we define the Toeplitz operator $T_g$ on
$\mathcal{H}_t(\mathbb{C}^n)$ by
$$T_gf(z) = \int_{\C^n} g(w)f(w)K_s(z,w) d\mu_s(w).$$

By the condition
($\ast$), it is easy to see that $T_g$ is a densely defined operator on $\mathcal{H}_t(\mathbb{C}^n)$. Another important consequence of ($\ast$) is that $g\ast q_s$ is well defined for $0 < s < \frac{1}{2}\sinh4t$, where
$q_s(x) = (4\pi s)^{-\frac{n}{2}}e^{-\frac{1}{4s}|x|^2}$ is the heat kernel corresponding to the standard Laplacian on $\R^n $. In fact, it is a
$\mathcal{C}^{\infty}$ function on $\mathbb{C}^n$.
By using the semigroup property we get
$$g\ast q_{s+r}= (g\ast q_r)\ast q_s,$$
when $0 < s+r < \frac{1}{2}\sinh4t$.
Now we find some necessary and sufficient conditions on $g$ such that
$T_g$ is a bounded operator. These conditions are given in terms of
$g\ast q_s$ for $0< s < \frac{1}{2}\sinh4t$. In order to do this we transfer
$T_g$ to $L^2(\mathbb{R}^n)$ and find the corresponding Weyl symbol of the
resulting operator.

Following Folland \cite{Fo} we define the Weyl pseudo-differential operator
on $L^2(\mathbb{R}^n)$ with symbol $\sigma\in $
$\mathcal{S}'(\mathbb{R}^{2n})$ by
\begin{equation}\label{38}
\sigma(D,X)f(x) = (2\pi)^{-n}\int_{\mathbb{R}^n} \int_{\mathbb{R}^n}\sigma (\frac{1}{2}(x+y),\xi)e^{-i(x-y).\xi} f(y)
dy d\xi.\end{equation}
We recall that (see \cite{Fo}) $\sigma(D,X)= W(\hat{\sigma})$, where $W$ is
the Weyl transform and $\hat{\sigma}$ is the Fourier transform of a
tempered distribution. We define for $\sigma \in \mathcal{S}'(\mathbb{R}^{2n})$
\begin{equation}
\sigma_t(x,\xi) = \sigma(\cosh(2t)x,- \sinh(2t)\xi).
\end{equation}
Note that $\sigma \rightarrow \sigma_t$ is an isomorphism on
$\mathcal{S}'(\mathbb{R}^{2n})$.

\begin{thm}
Let $T_g$, defined as above, be bounded. Then we have
$\|g\ast q_s\|_{\infty} \leq c(s)\|T_g\|$ for all s $\in (\frac{1}{8}\sinh4t,\frac{1}{2}\sinh4t )$. Conversely, if we assume that
$\|g\ast q_s\|_{\infty} < \infty $ for some
$0 < s < \frac{1}{8}\sinh4t$, then
$T_g$ is bounded. Moreover, we have
$$ \|T_g\| \leq c(s) \|g\ast q_s\|_{\infty}.$$
\end{thm}
\begin{proof} First let us assume that $T_g$ is bounded.
For $\frac{1}{4}\sinh4t\leq s < \frac{1}{2}\sinh4t $ the proof is trivial. We look at the Berezin Transform of $T_g$ defined by (see \cite{Fo})
\begin{equation}\label{1.1}
\tilde{T_g}(z)= \langle T_g k_z, k_z\rangle_{\mathcal{H}_t}.
\end{equation}
It is easy to check that $\tilde{T_g}(z)= g \ast q_{\frac{1}{4}\\sinh4t}(z)$. Here $k_z(w)= \frac{K_t(w,z)}{\sqrt{K_t(z,z)}}$ is the
normalized reproducing kernel.
In fact, even if $T_g$ is not bounded the Berezin transform is well defined because of the condition ($\ast$) and it is the same as above.
By applying Cauchy-Schwarz inequality to (\ref{1.1}) we get
\begin{equation}\label{1.2}
|g \ast q_{\frac{1}{4}\sinh4t}(z)| \leq \|T_g\|, ~~~z \in \mathbb{C}^n.
\end{equation}
So, by the semigroup property, when $ 0 < s < \frac{1}{2}\sinh4t$ we get
$g \ast q_{s+\frac{1}{4}\sinh4t}(z)$ = $(g \ast q_{\frac{1}{4}\sinh4t})\ast q_s(z)$ and
\begin{equation}\label{1.3}
\|g \ast q_{s+\frac{1}{4}\sinh4t}\|_{\infty}\leq c(s)\|T_g\|,
\end{equation}
where $c(s)$ is independent of $g$.
For proving the estimate for the other half of the interval in the statement
of the theorem, we make use of the boundedness of the operator
$ e^{tH}T_ge^{-tH} $ on $L^2(\mathbb{R}^n).$ Let $ e^{tH}T_ge^{-tH} =
W(\hat{\sigma_t}) $ for some $ \sigma \in \CS'(\R^{2n}).$

In order to find the
explicit form of $ \sigma_t $ we calculate the Berezin transform of $ T_g $
in terms of $ \sigma.$ By using (\ref{38}) an easy computation shows that
$$ \langle T_gk_z,k_z \rangle_{\CH_t} = \langle e^{-tH}\sigma_t(D,X)e^{tH}k_z,k_z \rangle_{\CH_t},$$
\begin{equation}\label{a}
\langle \sigma_t(D,X)e^{tH}k_z,e^{tH}k_z\rangle_{L^2(\R^n)} = \sigma \ast q_{\frac{1}{8}\sinh 4t}(z).
\end{equation}
By equating (\ref{1.1}) and (\ref{a}) we get
\begin{equation}\label{xy}g \ast q_{\frac{1}{4}\sinh 4t}(z)=
\sigma \ast q_{\frac{1}{8}\sinh 4t}(z), ~~~ z\in\C^n.
\end{equation}
Given that $g$ satisfies ($\ast$) and $\sigma\in \CS'(\R^{2n}),$ for a fixed
$z \in \mathbb{C}^n$ it is easy to check the following two facts :
(i) $ s\longrightarrow g\ast q_s(z)$ extends as a holomorphic function to
the domain
$$D_1 = \{\zeta\in\mathbb{C}:|\zeta-\frac{1}{4}\sinh4t|< \frac{1}{4}\sinh4t\}$$
and (ii) $s\longrightarrow \sigma\ast q_s(z)$ extends as a holomorphic
function to $ D_2=\{\zeta\in \mathbb{C}: \Re\zeta >0\}$.
By using the above two facts we get that $g*q_{\frac{1}{8}\sinh 4t}*q_s\equiv\sigma \ast q_s$ for all $0<s<\frac{3}{8}\sinh4t$. Now taking the limit
$ s \longrightarrow 0 $ we get
$$\sigma_t(x,y)=g*q_{\frac{1}{8}\sinh 4t}(\cosh(2t)x,-\sinh(2t)y).$$

Using the fact that $\CB(L^2(\R^n)) $ is the dual of the space of all
trace class operators, we get the following:
\begin{equation}\label{1.6}
|tr (W(f)W(\hat{\sigma_t}))|\leq \|T_g\| \|W(f)\|_{tr}
\end{equation} for all $f\in L^2(\mathbb{C}^{n})$ such that $W(f)$ is trace class. In particular, (\ref{1.6}) holds for all $f$ in Schwartz class.
It is easy to compute that
\begin{equation}\label{1.7}
tr (W(\bar{z})W(f)W(\bar{z})^{\ast}W(\hat{\sigma_t}))=
\hat{f}\ast\sigma_t(z)
\end{equation}
for all $z$ when $f\in \mathcal{S}(\mathbb{R}^{2n})$ and $\sigma \in \mathcal{S}'(\mathbb{R}^{2n})$.
If we choose $f$ in (\ref{1.7}) such that
$\hat{f}(w)= q_{t_{1}}(u)q_{t_{2}}(v), w=u+iv $ where $t_{1}= s\cosh2t $,
$t_{2}=s\sinh2t$ and $z=(\cosh2t)^{-1}x+i(\sinh2t)^{-1}y $ we get
$$ tr (W(\bar{z})W(f)W(\bar{z})^{\ast}W(\hat{\sigma_t}))=
\sigma\ast q_s(x+iy) $$
for all $ s > 0.$ By (\ref{1.6})
$$|\sigma\ast q_s(x+iy)|\leq c(s)\|T_g\|,$$ where
$\sigma=g\ast q_{\frac{1}{8}\sinh4t}$ and this implies
\begin{equation}
|(g\ast q_{\frac{1}{8}\sinh4t})\ast q_s(x+iy)|\leq c(s)\|T_g\|
\end{equation}
for all $ z\in \mathbb{C}^n.$ Finally, the boundedness of $T_g$ implies
that
$$\|g\ast q_s\|_{\infty}\leq c(s)\|T_g\| $$
whenever $ s>\frac{1}{8}\sinh 4t.$

Conversely, let $\|g\ast q_s\|_{\infty} < \infty$ for some $0 < s < \frac{1}{8}\sinh4t$
then proceeding as in Berger and Coburn \cite{BC}
$$ \|\sigma_t\|_{\ast} \equiv \Sigma_{|\mu| +|\beta|\leq 2n+1}
\|D^{\mu}_{\xi}D^{\beta}_{x}\sigma_t\|_{\infty} < \infty,$$
where $\sigma=g\ast q_{\frac{1}{8}\sinh 4t}$. Now we can appeal to
Theorem 2.73 in \cite{Fo} by which $\sigma_t(D,X)$ is bounded with
$\|\sigma_t(D,X)\|\leq \|\sigma_t\|_{\ast}.$ The Berezin symbol of
$e^{-tH}\sigma_t(D,X)e^{tH}$ (see (\ref{a}) and (\ref{1.1})) is given by
$$ (e^{-tH}\sigma_t(D,X)e^{tH})\tilde{}(z)
= \sigma\ast q_{\frac{1}{8}\sinh 4t}(z)=g\ast q_{\frac{1}{4}\sinh 4t}(z)$$
which implies that
$$\tilde{T_g}\equiv (e^{-tH}\sigma_t(D,X)e^{tH})\tilde{}.$$
Hence by the uniqueness of the Berezin transform
$T_g=e^{-tH}\sigma_t(D,X)e^{tH}.$ Therefore, the boundedness of
$ \sigma_t(D,X) $ implies that $\|T_g\|\leq\|\sigma_t\|_{\ast}$. As  shown
in \cite{BC} we have
$\|\sigma_t\|_{\ast}\leq c(n,s)\|g\ast q_s\|_{\infty}$.
Hence the theorem is proved.
\end{proof}

\begin{rem} The above theorem is the analogue of Theorems 11 and 12 
in \cite{BC}. As
in \cite{BC} we conjecture that $ T_g $ is bounded if and only if $
g*q_{\frac{1}{8}\sinh 4t} $ is bounded. We have a class of symbols supporting
this conjecture, see Section 3.3
\end{rem}

\subsection{Hermite multipliers and Toeplitz operators}

In this subsection we are interested in finding a necessary and sufficient
condition on the symbol $ g $ so that $ e^{tH}T_ge^{-tH} $ is a Hermite
multiplier.
 Using the fact that $ W(\sigma) $ is a function of the
Hermite operator if and only if the symbol $ \sigma $ is a radial
distribution we get the following result.
\begin{thm} Given $ T_g $ on $ \CH_t(\C^n) $ the operator $e^{tH}T_ge^{-tH} $
is a Hermite multiplier if and only if $
g*q_{\frac{1}{8}\sinh(4t)}(\cosh(2t)y,-\sinh(2t)v) $ is a radial function
on $ \R^{2n}.$
\end{thm}

\begin{cor} Let $ g $ be as in the theorem. Then $ T_g $ is bounded on
$\CH_t(\C^n) $ if and only if the sequence
$$ \frac{k!(n-1)!}{(k+n-1)!}\int_{\R^{2n}}
g*q_{\frac{1}{8}\sinh(4t)}(\cosh(2t)y,-\sinh(2t)v)\varphi_k(2y,2v) dy dv $$
is bounded.
\end{cor}

\begin{exam}
An example of symbol satisfying the condition given in Theorem 3.3
is provided by $ g(y,v) = e^{\alpha |y|^2+\beta |v|^2} $ under suitable
conditions on $ \alpha $ and $ \beta.$  A simple calculation shows that
$$g*q_{\frac{1}{8}\sinh(4t)}(\cosh(2t)y,-\sinh(2t)v) =
e^{\frac{\alpha \coth(2t)\sinh(4t)}{2-\alpha \sinh(4t)}|y|^2}
e^{\frac{\beta \tanh(2t)\sinh(4t)}{2-\beta \sinh(4t)}|v|^2} $$
and hence $ g*q_{\frac{1}{8}\sinh(4t)}(\cosh(2t)y,-\sinh(2t)v)$ is radial
if and only if
$$ \frac{\alpha \coth(2t)}{2-\alpha \sinh(4t)} =
\frac{\beta \tanh(2t)}{2-\beta \sinh(4t)}.$$ After simplification we get the
condition $ \alpha \coth(2t)-\beta \tanh(2t) = \alpha \beta $ which is
necessary and sufficient for the radiality of the function
$$ g*q_{\frac{1}{8}\sinh(4t)}(\cosh(2t)y,-\sinh(2t)v),~~~
g(y,v) = e^{\alpha |y|^2+\beta |v|^2} .$$

When the above condition is verified, by Corollary 3.3 the operator $ T_g $
is bounded on $ \CH_t(\C^n) $ if and only if the sequence
$$ \frac{k!(n-1)!}{(k+n-1)!}\int_{\R^{2n}}
e^{\frac{\alpha \coth(2t)\sinh(4t)}{2-\alpha \sinh(4t)}(|y|^2+|v|^2)}
\varphi_k(2y,2v) dy dv $$
is bounded. Again, by repeating the method in the Theorem 1.1 this is equivalent to the boundedness of the sequence
$$ \frac{k!(n-1)!}{(k+n-1)!} \int_{\C^n} e^{\lambda|z|^2}
\varphi_k(2z) dz $$
$$= c_n \int_{\C^n} e^{\frac{\lambda}{1+\lambda}|z|^2}
\frac{|z|^{2k}}{k!2^k} e^{-\frac{1}{2}|z|^2} dz = c_n \left(\frac{1+\lambda}{1-\lambda}\right)^k$$
where $ \lambda =   \frac{\alpha \coth(2t)\sinh(4t)}{2-\alpha \sinh(4t)}.$
Thus the condition for the
boundedness of $ T_g $ reduces to $ |\frac{1+\lambda}{1-\lambda}| \leq 1 $ or
$ \Re \lambda \leq 0.$ In terms of $ \alpha $ the condition reads as
$ |\alpha|^2 \sinh(4t)-2\Re \alpha \geq 0.$ So, the necessary and sufficient condition for $T_g$ to be bounded is the boundedness of $g*q_{\frac{1}{8}\sinh (4t)}.$ It is worth comparing this
example with a similar example given in \cite{BC}.
\end{exam}

The condition in Corollary 3.3 on $ g $ is not easy to check. However, using Gutzmer's
formula we can get a sufficient condition in a more convenient form for
certain special class of symbols. Consider
radial functions $ h(y,v) $ on $ \R^{2n} $ for which
\begin{equation}\label{26}
 \int_{\R^{2n}} h(y,v)e^{|y|^2+|v|^2} (|y|^2+|v|^2)^{k/2} < \infty
\end{equation}
for all $ k \in \N.$ Define a function $ g $ by the equation
\begin{equation}\label{27}
g(\xi,v)U_t(\xi,v) = \int_{\R^n} e^{-2y\cdot \xi} h(y,v) dy.
\end{equation}

\begin{thm}\label{28} Suppose $ g $ is given by (\ref{27}) where $ h $ 
satisfies (\ref{26}). Then
we have $ e^{tH}T_ge^{-tH} = m_t(H) $  where
$$ m_t(2k+n) = e^{-2(2k+n)t} \frac{k!(n-1)!}{(k+n-1)!}
\int_{\R^{2n}} h(y,v) \varphi_{k}(2iy,2iv) dy dv .$$ Consequently,
$ T_g $ is bounded on $ \CH_t(\C^n) $ if and only if
$ |m_t(2k+n)| \leq C $  for all $ k \in \N.$
\end{thm}
\begin{proof} As we mentioned we prove this theorem by using Gutzmer's formula.
Indeed, polarising Gutzmer's formula we obtain
$$ \int_K \int_{\R^n} \pi(k.(iy,iv))F_1(\xi)\overline{\pi(k.(iy,iv))F_2(\xi)}
d\xi dk $$
$$ =\sum_{k=0}^\infty  \frac{k!(n-1)!}{(k+n-1)!}e^{-2t(2k+n)}\langle P_kf_1,f_2\rangle
\varphi_k(2iy,2iv) $$
where  $ F_j = e^{-tH}f_j, j = 1,2 $ are from $ \CH_t(\C^n).$ Integrating the
above identity with respect to $ h(y,v)dy dv $ we obtain
$$ \int_{\R^{2n}} \int_K \int_{\R^n} \pi(k.(iy,iv))F_1(\xi)\overline
{\pi(k.(iy,iv))F_2(\xi)} h(y,v) d\xi dk dy dv $$
$$ = \sum_{k=0}^\infty  \frac{k!(n-1)!}{(k+n-1)!}e^{-2t(2k+n)}\langle P_kf_1,f_2\rangle
\int_{\R^{2n}}h(y,v)\varphi_k(2iy,2iv) dy dv. $$
When the function $ h $ is $ K$ invariant,
$$ \int_{\R^{2n}} \int_K \int_{\R^n} \pi(k.(iy,iv))F_1(\xi)\overline
{\pi(k.(iy,iv))F_2(\xi)} h(y,v)d\xi dk dy dv $$
$$ = \int_{\R^{2n}} \int_{\R^n} \pi(iy,iv)F_1(\xi)\overline
{\pi(iy,iv)F_2(\xi)} h(y,v)d\xi dy dv .$$ Recalling the
definition of $ \pi(iy,iv) $ the above integral can be rewritten as
$$   \int_{\R^{2n}} \int_{\R^n}F_1(\xi+iv) \overline{F_2(\xi+iv)}
e^{-2 y\cdot \xi} h(y,v)d\xi dv.$$

Suppose now $ g(\xi,v) $ satisfies the equation
$$ g(\xi,v) U_t(\xi,v) = \int_{\R^n} e^{-2 y\cdot \xi} h(y,v)dy $$
and $ m_t(2k+n) $ is defined by
$$ m_t(2k+n) = e^{-2t(2k+n)} \frac{k!(n-1)!}{(k+n-1)!}\int_{\R^{2n}}h(y,v)\varphi_k(2iy,2iv) dy dv. $$
Then it is clear that we have obtained
$$  \int_{\R^{2n}} F_1(\xi+iv) \overline{F_2(\xi+iv)} g(\xi,v) U_t(\xi,v)
d\xi dv $$
$$ = \sum_{k=0}^\infty m_t(2k+n)\langle P_kf_1,f_2\rangle $$
which simply means
that $ (T_gF_1,F_2)_{\CH_t} = \langle m_t(H)f_1,f_2\rangle_{L^2} $ where
$$ m_t(H)f = \sum_{k=0}^\infty m_t(2k+n) P_kf $$ is the Hermite multiplier.
Thus the boundedness of the Toeplitz operator $ T_g $ on $ \CH_t(\C^n) $ is
equivalent to the boundedness of $ m_t(H) $ on $ L^2(\R^n).$
\end{proof}

\begin{rem}
In the above proof of sufficiency we have not used Theorem 3.3  but the 
condition
stated in that theorem can be verified. Indeed, when $ g $ satisfies the
equation (\ref{27}) a simple calculation shows that
$$ g*q_{\frac{1}{8}\sinh(4t)}(\cosh(2t)x,-\sinh(2t)y)
e^{\tanh(2t)(|x|^2+|y|^2)}$$
$$ = \int_{\R^{2n}} h(\xi,v)e^{\tanh(2t)(|\xi|^2+|v|^2)}e^{-\frac{2}{\cosh(2t)}
(x\cdot \xi+y\cdot v)} d\xi dv $$
from which it is clear that
$ g*q_{\frac{1}{8}\sinh(4t)}(\cosh(2t)x,-\sinh(2t)y)$ is radial whenever
$ h(\xi,v) $ is radial. The above equation also suggests a relation between
$ g $ and $ h.$
\end{rem}

\begin{rem} The radiality of the function
$$ g*q_{\frac{1}{8}\sinh(4t)}(\cosh(2t)x,-\sinh(2t)y) $$ is not equivalent
to the factorisation given in (\ref{27}).
 Indeed, consider the symbol $ g(x,y) =
e^{\alpha |x|^2 +\beta |y|^2} $ considered earlier with the conditions
$ \alpha \coth(2t)-\beta \tanh(2t) = \alpha \beta $ and $ \Re(\alpha) <
\frac{1}{2}(\sinh(2t))^{-1}.$ If there exists a function $ h $ such that
$$ g(\xi,v)U_t(\xi,v) = \int_{\R^n} e^{-2\xi \cdot y}h(y,v) dy, $$ then
we have the relation
$$ g(\frac{i}{2}\xi,v)U_t(\frac{i}{2}\xi,v) = \int_{\R^n} e^{-i\xi \cdot y}
h(y,v) dy .$$
This leads to the equation
$$ e^{-\frac{1}{4}(\tanh(2t)+\alpha)|y|^2} e^{-(\coth(2t)-\beta)|v|^2} =
\int_{\R^n} e^{-i\xi \cdot y} h(y,v) dy .$$
By Fourier inversion we see that
$$ h(y,v) = c~ e^{-\frac{1}{\tanh(2t)+\alpha}|y|^2} e^{-(\coth(2t)-\beta)|v|^2}
$$ which is not radial in general.
\end{rem}

\begin{rem}
Since $$ U_t(\xi,v) = c_t \int_{\R^n} p_{2t}(2y,2v)e^{-2y \cdot \xi} dy $$ the
equation (\ref{27}) is equivalent to
$$ g(\xi,v) = \int_{\R^n} g_1(y,v)e^{-2y \cdot \xi} dy .$$ Indeed, if $ g $
satisfies the above equation, then the function
$$ h(y,v) = \int_{\R^n} g_1(y-u,v)p_{2t}(2u,2v) du $$ satisfies
$$ \int_{\R^n} h(y,v) e^{-2y\cdot \xi} dy = g(\xi,v) U_t(\xi,v) $$ as can be easily verified. Thus for such symbols  Theorem \ref{28} is valid.
\end{rem}

\section{Toeplitz operators on Twisted Bergman spaces}
\setcounter{equation}{0}

In this section we take up the study of Toeplitz operators on twisted Bergman
spaces which are Segal-Bargmann spaces associated to the special Hermite
semigroup $ e^{-tL}.$ These spaces arise naturally in the study of
Segal-Bargmann transform on the Heisenberg group, see \cite{KTX1}. We show
that $ e^{tL}T_g e^{-tL} $ is a Weyl multiplier if and only if  the symbol
$ g(x+iy,u+iv) $ depends only on $ (y,v).$  By means of Gutzmer's formula
we study boundedness of $ T_g $ which correspond to multipliers for special
Hermite operators.

\subsection{Twisted Bergman spaces}
By the term twisted Bergman spaces we mean the Hilbert space of entire
functions $ F(z,w) $ on $ \C^{2n} $ which are square integrable with respect
to the weight function
$$ W_t(z,w) = e^{(u\cdot y-v\cdot x)}p_{2t}(2y,2v) $$
where
$$ p_{t}(y,v) = c_n(\sinh(2t))^{-n} e^{-\frac{1}{4}\coth(2t)(|y|^2+|v|^2)} $$
is the heat kernel associated to the special Hermite operator $ L ,$ see \cite{Th3}.
Thus the special Hermite semigroup $ e^{-tL} $ is given by $  e^{-tL}f =
f \times p_t ,$ the twisted convolution of $ f $ with $ p_t.$ These spaces,
denoted by $ \CB_t^*(\C^{2n}) ,$ arise naturally in the study
of Segal-Bargmann transform on the Heisenberg group \cite{KTX1}. The following result
proved in \cite{KTX1} characterises $ \CB_t^*(\C^{2n}) $.

\begin{thm} An entire function $ F $ on $ \C^{2n} $ belongs to
$ \CB_t^*(\C^{2n})$ if and only if its restriction to $ \R^{2n} $ is of the
form $ e^{-tL}f(x,u) $ for some $ f \in L^2(\R^{2n}).$ Moreover, the norm of
$ F $ in $ \CB_t^*(\C^{2n})$ is the same as the norm of $ f $ in
$ L^2(\R^{2n}).$
\end{thm}

Another proof of this was found in \cite{Th2} which is based on the following
Gutzmer's formula for the special Hermite expansion. Recall that $ \varphi_k(x,u) = \varphi_k(x+iu) $ are the
Laguerre functions of type $ (n-1) $ introduced earlier. They extend to
entire functions on $ \C^{2n} $ and are denoted by $ \varphi_k(z,w).$ The
twisted convolutions $ f \times \varphi_k $ are the projections onto the
$ k$-th eigenspace of $ L $ and the special Hermite expansion of an $ f
\in L^2(\R^{2n}) $ is written as
$$ f = (2\pi)^{-n} \sum_{k=0}^\infty f\times \varphi_k $$
where the series converges in $ L^2.$ The heat kernel $p_t$ associated to the special Hermite operator can also be written as
$$p_t(x,u)= (2\pi)^{-n}\sum_{k=0}^{\infty}e^{-(2k+n)t}\varphi_k(x,u).$$

\begin{thm}For any $F\in\mathcal{O}(\mathbb{C}^{2n})$ we have
$$\int_{\mathbb{R}^{2n}}\int_{K}e^{(u.y-v.x)}|F(\sigma(x+iy,u+iv))|^2 dxdu
d\sigma $$
$$=\sum_{k=0}^{\infty}\frac{k!(n-1)!}{(k+n-1)!}\|f\times\varphi_k\|_2^2
~\varphi_k(2iy,2iv),$$ where $f$ is the restriction of $F$ to $\R^{2n}.$
\end{thm}

Clearly when $F= e^{-tL}f$ the above formula holds.
So, Theorem 4.1 easily follows from Gutzmer's formula once we have the identity
\begin{equation}\label{37} \int_{\R^{2n}} p_{2t}(2y,2v)\varphi_k(2iy,2iv) dydv = \frac{(k+n-1)!}{k!
(n-1)!}e^{(2k+n)2t}.
\end{equation}
This has been proved in \cite{Th2}, see Lemma $6.3$.  The following extension of this
result is needed for the study of Toeplitz operators.

\begin{lem}
\begin{equation} \label{4.1}\int_{\R^{2n}} e^{\frac{i(u\cdot y-v\cdot x)}{2}} p_{t}(x-y,u-v)
\varphi_k(iy,iv) dydv = \varphi_k(ix,iu) e^{(2k+n)t}.\end{equation}
\end{lem}
\begin{proof} Recall that the symplectic Fourier transform $ \tilde{f} $ of a
function $ f $ is defined by $ \tilde{f}(x,y) = \hat{f}(\frac{1}{2}(-y,x)).$
We know that $\varphi_k$'s are eigenfunctions of the symplectic Fourier
transform with  eigenvalues $(-1)^k$, i.e.
$$ (2\pi)^{-n} \int_{\mathbb{R}^{2n}}\varphi_k(\xi,\eta)
e^{\frac{i(\eta.y-\xi.v)}{2}}d\xi d\eta= (-1)^k\varphi_k(y,v).$$
The above equation remains true even if we replace $ (y,v) $ by $ (iy,iv) $.
So we get
\begin{equation}\label{4.2}
(2\pi)^{-n}\int_{\mathbb{R}^{2n}}\varphi_k(\xi,\eta)e^{-\frac{(\eta.y-\xi.v)}{2}}d\xi d\eta= (-1)^k\varphi_k(iy,iv).
\end{equation}
Now putting (\ref{4.2}) in (\ref{4.1}) and by using Fubini's theorem we get
$$(-1)^k\int_{\mathbb{R}^{2n}}\int_{\mathbb{R}^{2n}}e^{\frac{i(u\cdot y-v\cdot x)}{2}} e^{-\frac{\coth t }{4}(x-y)^2+(u-v)^2}e^{-\frac{(\eta.y-\xi.v)}{2}}
\varphi_k(\xi,\eta) dy dvd\xi d\eta$$
$$=(-1)^k\int_{\mathbb{R}^{2n}}e^{\frac{\eta.x-\xi.u}{2}}e^{-\frac{\tanh t}{4}(u-i\eta)^2+(x-i\xi)^2}\varphi_k(\xi,\eta)d\xi d\eta.$$
Now look at the function
$$F(t)=\int_{\mathbb{R}^{2n}}e^{\frac{\eta.x-\xi.u}{2}}e^{-\frac{\tanh t}{4}(u-i\eta)^2+(x-i\xi)^2}\varphi_k(\xi,\eta)d\xi d\eta.$$
If we replace $t$ by $z $ with $|\Im{z}|<\pi/2$, it is easy to see that the integral converges absolutely. In fact, $F$ can be extended as a holomorphic function to the strip $|\Im{z}|<\pi/2 $ containing the real line.
Consider
$$F(-t)=\int_{\mathbb{R}^{2n}}e^{\frac{\eta.x-\xi.u}{2}}e^{\frac{\tanh t}{4}(u-i\eta)^2+(x-i\xi)^2}\varphi_k(\xi,\eta)d\xi d\eta $$
which after using
$$ e^{\frac{\tanh t}{4}(u-i\eta)^2+(x-i\xi)^2}=
e^{-\frac{\tanh t}{4}(\eta+iu)^2+(\xi+ix)^2}$$
reads as $$F(-t)=\int_{\mathbb{R}^{2n}}e^{\frac{\eta.x-\xi.u}{2}}e^{-\frac{\tanh t}{4}(\eta+iu)^2+(\xi+ix)^2}\varphi_k(\xi,\eta)d\xi d\eta .$$
This is nothing but the twisted convolution of $\varphi_k$ with
$\widetilde{p_t}$ at $(iu,ix).$ It is easy to calculate $F(-t)$ by recalling
$$\widetilde{p_t}(\xi,\eta)=\sum_{k=0}^{\infty}e^{-(2k+n)t}(-1)^k
\varphi_k(\xi,\eta).$$
Using the above along with the fact that $\varphi_k\times\varphi_j=
c_n\delta_{j,k}\varphi_k $ we get $$F(-t)=(-1)^ke^{-(2k+n)t}
\varphi_k(ix,iu).$$ The right hand side of the above equation is also a
holomorphic function of $t$ and both sides agree on the negative real axis.
Therefore, they agree everywhere and changing $ t $ into $ -t $ we get the
lemma.
\end{proof}

Note that when $ x = u =0 $ this lemma reduces to (\ref{37}) as
$ \varphi_k(0,0) =  \frac{(k+n-1)!}{k!(n-1)!}.$

\subsection{ Toeplitz operators and special Hermite multipliers}

In this subsection we get some necessary and sufficient conditions for the boundedness
of $T_g$ on $\CB_t^*(\C^{2n})$ for a special class of symbols, by making use of Gutzmer's
Formula for the special Hermite expansion.
First note that by Theorem 4.1 $e^{-tL}\Phi_{\alpha,\beta}(z,w)= 
e^{-(2|\beta|+n)t}\Phi_{\alpha,\beta}(z,w)$
form an orthonormal basis for $\CB_t^*(\C^{2n}).$ We denote $e^{-tL}\Phi_{\alpha,\beta}$ by $\phi_{\alpha,\beta}$
in this section.
Consider a measurable function $g$ on $\mathbb{C}^{2n}$ for which
\begin{equation}\label{99}\int_{\mathbb{C}^{2n}}|g(z,w)\phi_{\alpha,\beta}(z,w)\phi_{\mu,\nu}(z,w)|W_t(z,w)dz dw <\infty.
\end{equation}
Now we can define a densely defined bilinear form on $\CB_t^*(\C^{2n})$ by
$$\langle T_g \phi_{\alpha,\beta},\phi_{\mu,\nu}\rangle_{\CB_t^*(\C^{2n})} := \int_{\mathbb{C}^{2n}}g(z,w)\phi_{\alpha,\beta}(z,w)\overline{\phi_{\mu,\nu}}(z,w)W_t(z,w)dz dw.$$
These are nothing but the matrix entries of $T_g$.
We consider special symbols $g$ for which the above densely defined bilinear form becomes a diagonal form. Such symbols are provided by functions of the 
form  $ g(x+iy,u+iv)= g_0(y,v)$ where $g_0 $ is a radial function on 
$\mathbb{R}^{2n}.$
\begin{thm}
Let $g $ be as above and satisfy (4.2.4). Then $T_g $ is bounded if and only 
if the sequence
$$ e^{-(2k+n)2t} \frac{k!(n-1)!}{(k+n-1)!}\int_{\mathbb{R}^{2n}}g_0(y,v)
p_{2t}(2y,2v)\varphi_k(2iy,2iv) dy dv  $$ is  bounded.
\end{thm}
\begin{proof}
Clearly $\langle T_g \phi_{\alpha,\beta},\phi_{\mu,\nu}\rangle_{\CB_t^*(\C^{2n})}$
is well defined for all $ (\alpha,\beta)$ and $(\mu,\nu)$.
As done in Section 3.3 we can polarize Gutzmer's formula to obtain
$$\int_{\mathbb{R}^{2n}}\int_{K}e^{(u.y-v.x)}e^{-tL}f_1(\sigma(x+iy,u+iv))\overline{e^{-tL}f_2}(\sigma(x+iy,u+iv)) dx du d\sigma
$$ $$=\sum_{k=0}^{\infty}\frac{k!(n-1)!}{(k+n-1)!}e^{-(2k+n)2t}\langle f_1\times\varphi_k,f_2\times\varphi_k\rangle_{L^2(\mathbb{R}^{2n})}\varphi_k(2iy,2iv).$$
When $ f_1=\Phi_{\alpha,\beta}$ and $ f_2=\Phi_{\mu,\nu}$ the above identity
reduces to
\begin{equation}\label{c}
\int_{\mathbb{R}^{2n}}\int_{K}e^{(u.y-v.x)}
\phi_{\alpha,\beta}(\sigma(z,w))\overline{\phi_{\mu,\nu}}(\sigma(z,w))
dx du d\sigma
\end{equation}
$$=\sum_{j=0}^{\infty}\frac{j!(n-1)!}{(j+n-1)!}
e^{-(2j+n)2t}\langle \Phi_{\alpha,\beta}\times\varphi_j,\Phi_{\mu,\nu}
\times\varphi_j\rangle_{L^2(\mathbb{R}^{2n})}\varphi_j(2iy,2iv)$$
$$ =\delta_{\alpha,\mu}\delta_{\beta,\nu}\frac{k!(n-1)!}{(k+n-1)!}
e^{-(2k+n)2t}\varphi_k(2iy,2iv),$$
where $|\beta|=k.$
Writing the matrix coefficients explicitly
$$\langle T_g \phi_{\alpha,\beta},\phi_{\mu,\nu}\rangle_{\CB_t^*(\C^{2n})} =\int_{\mathbb{C}^{2n}}g_0(y,v)\phi_{\alpha,\beta}(z,w)\overline{\phi_{\mu,\nu}}(z,w)W_t(z,w)dz dw.$$
The above integral converges absolutely. Now, replace $(z,w)$ by $\sigma(z,w)$ where $\sigma\in K$.
Since $g_0(y,v)W_t(z,w)dz dw$ is invariant under the action of $ K $ we get
$$\langle T_g \phi_{\alpha,\beta},\phi_{\mu,\nu}\rangle_{\CB_t^*(\C^{2n})}$$
$$=\int_K\int_{\mathbb{C}^{2n}}g_0(y,v)\phi_{\alpha,\beta}(\sigma(z,w))
\overline{\phi_{\mu,\nu}}(\sigma(z,w))W_t(z,w)dz dw d\sigma.$$
The integral converges absolutely and hence  by applying Fubini's theorem
and using (\ref{c}) we get
$$\langle T_g \Phi_{\alpha,\beta},\Phi_{\mu,\nu}\rangle_{\CB_t^*(\C^{2n})}$$
$$ =\delta_{\alpha,\mu}\delta_{\beta,\nu}\frac{k!(n-1)!}{(k+n-1)!}
e^{-(2k+n)2t}\int_{\mathbb{R}^{2n}}g_0(y,v)p_{2t}(2y,2v)\varphi_k(2iy,2iv)dydv,$$
where $|\beta|=k$. Thus the operator $ T_g $ is diagonal in the basis 
$$ \{ \phi_{\alpha,\beta} : \alpha,\beta \in \N^n \} $$ 
and the theorem follows.
\end{proof}

Let $h$ be a radial measurable function on $\mathbb{R}^{2n}$ and assume that
\begin{equation}\label{63}
\int_{\mathbb{R}^{2n}}|h(y,v)|e^{s(|y|^2+|v|^2)}dy dv <\infty
\end{equation}
for all $ s > 0.$ Consider the symbol defined by
$$ g(x+iy,u+iv)p_{2t}(2y,2v)= h\times p_{2t}(2y,2v).$$

\begin{cor}In the above theorem let $ g $ be as above with  $h$ satisf
ying (\ref{63}). Then $T_g$ is bounded if and only if the sequence
$$\frac{k!(n-1)!}{(k+n-1)!}\int_{\mathbb{R}^{2n}}h(y,v)\varphi_k(iy,iv)dydv $$
is bounded.
\end{cor}
\begin{proof} As $ h $ and $ p_t $ are both  radial, so is $ g_0(y,v) = 
g(iy,iv).$ Hence by
Theorem 4.4 we know that $T_g $is bounded if and only if
$$ \frac{k!(n-1)!}{(k+n-1)!}e^{-(2k+n)2t}\int_{\mathbb{R}^{2n}}g_0(y,v)
p_{2t}(2y,2v)\varphi_k(2iy,2iv) dy dv.$$
As $g(iy,iv)p_{2t}(2y,2v)= h\times p_{2t}(2y,2v)$ the above simplifies to
\begin{equation}\label{65}
\frac{k!(n-1)!}{(k+n-1)!}e^{-(2k+n)2t}\int_{\mathbb{R}^{2n}}
h\times p_{2t}(2y,2v)\varphi_k(2iy,2iv) dy dv.
\end{equation}
Because of (\ref{63}) we can use Fubini's theorem to change the order of
integration. By Lemma 4.3 we have
$$ \int_{\mathbb{R}^{2n}}p_{2t}(x-2y,u-2v)e^{i(u.y-v.x)}\varphi_k(2iy,2iv)dydv
= e^{(2k+n)2t}\varphi_k(ix,iu).$$
Using this in  (\ref{65}) we obtain the corollary.
\end{proof}

From now on let us assume that $g$ is  a measurable function on
$\mathbb{C}^{2n}$ such that
$\int_{\mathbb{C}^{2n}}|g(z,w)\phi_{\alpha,\beta}(z,w)|^2 W_t(z,w)
dzdw < \infty $ for all $\alpha,\beta .$ We will refer to this condition as $(**).$
Note that the condition $(**)$ on $g$ implies that it belongs to
$$L^2(\C^{2n},e^{(u.y-v.x)}e^{-\frac{(|x|^2+|u|^2)}{2}}
e^{(-\coth 2t +\frac{1}{2})(|y|^2+|v|^2)}dzdw).$$
For such symbols the Toeplitz operator on $\CB_t^*(\C^{2n})$
is defined by $T_g(f):=P(gf)$ where  $P$ is the orthogonal projection
from $ L^2(\C^{2n}, W_t) $ onto $\CB_t^*(\C^{2n}).$  We study the class of
symbols $ g $ for which $ T_g $ is bounded and $ e^{tL}T_ge^{-tL}$ is a right Weyl
multiplier, i.e. $ W\left(e^{tL}T_ge^{-tL}f\right)= W(f)M_t $ for some 
$ M_t\in\mathcal{B}(L^2(\mathbb{R}^{2n})).$

Before proving the next theorem we prove a lemma which will be used. Let 
$$V_t(z,w)=e^{(u.y-v.x)}e^{-\frac{(|x|^2+|u|^2)}{2}}
 e^{(-\coth 2t +\frac{1}{2})(|y|^2+|v|^2)}$$
and consider the measure $d\tau(z,w) =V_t(z,w)dzdw$, where $dzdw$
is the Lebesgue measure on $\R^{4n}.$
Let $\CP(\R^{4n})$ be the set of all polynomials on $\R^{4n}.$ Note that  
$\CP(\R^{4n}) \subset L^2(\C^{2n},d\tau(z,w)).$
\begin{lem}$\CP(\R^{4n})$ is dense in $L^2(\C^{2n},d\tau(z,w)).$
\end{lem}
\begin{proof}  By abuse of notation let us denote any polynomial $ p \in 
\CP(\R^{4n})$ by $ p(z,w).$ As the weight function $V_t(z,w)$ corresponding 
to $d\tau$ is nowhere vanishing,
it is enough to show that the linear span of $p(z,w)\left(V_t(z,w)\right)^{1/2}$ is dense in
$L^2(\C^{2n}).$ More precisely,
if there exists $g\in L^2(\C^{2n})$ such that
\begin{equation}\label{100}\int_{\C^{2n}}g(z,w)p(z,w)\left(V_t(z,w)\right)^{1/2}dzdw=0\end{equation}
for all $p\in\CP(\R^{4n})$
then we need to show $g=0.$ Now suppose that there exists 
$g$ satisfying (\ref{100}).
It is easy to see that by completing the square in $ V_t(z,w) $ 
(\ref{100}) can be rewritten as
\begin{equation}\label{101}
\int_{\C^{2n}}g(z-v,w+y)p(z-v,w+y)e^{-\frac{1}{4}(|x|^2+|u|^2)}e^{-\frac{\coth2t-1}{2}(|y|^2+|v|^2)}dzdw=0
\end{equation}
for all $p\in\CP(\R^{4n})$. If we let $\tilde{g}(z,w)=g(z+v,w-y) $ then it is 
clear that $ \tilde{g} \in L^2(\C^{2n})$ whenever $g\in L^2(\C^{2n})$ and 
$\|\tilde{g}\|_{ L^2(\C^{2n})}=\|g\|_{ L^2(\C^{2n})}.$ So, it is enough to 
show that $\tilde{g} = 0.$ The equation (4.2.9) means that
$$\int_{\C^{2n}} \tilde{g}(z,w)q(z,w)e^{-\frac{1}{4}(|x|^2+|u|^2)}
e^{-\frac{\coth2t-1}{2}(|y|^2+|v|^2)}dzdw=0$$
for all $ q \in\CP(\R^{4n})$. As the linear span of functions of the form
$$ q(z,w)e^{-\frac{1}{4}(|x|^2+|u|^2)} e^{-\frac{\coth2t-1}{2}(|y|^2+|v|^2)}$$ is dense in $L^2(\C^{2n}) $ the last equation implies  $\tilde{g}=0 $ 
proving the lemma.
\end{proof}
\begin{thm} Let a Lebesgue measurable function $ g $ on $\C^{2n}$ satisfy
 $(**)$ and let $T_g$ be the corresponding Toeplitz operator on 
$\CB_t^{*}(\C^{2n}).$ Then $T_g=0$ if and only if $g=0$ a.e.
\end{thm}
\begin{proof}When $g=0$ a.e. clearly  $T_g=0.$
Conversely, let $T_g=0.$ We need to prove that $g=0$ a.e.
By  using the explicit form of the functions $ \phi_{\alpha,\beta}$ namely,
$\phi_{\alpha,\beta}(z,w)=P_{\alpha,\beta}(z,w)e^{-\frac{z^2+w^2}{4}},$ 
where $P_{\alpha,\beta}$ are  holomorphic polynomials on $\C^{2n}$ of degree 
$|\alpha|+|\beta|$ the condition $(**)$
takes the form
$$\int_{\C^{2n}}|g(z,w)P_{\alpha,\beta}(z,w)|^2V_t(z,w)dzdw<\infty$$
for all $\alpha,\beta.$ 
The above also implies that
$g\in L^2(\C^{2n},V_t(z,w)dzdw)$ in particular.
In view of the previous  lemma, proving $g=0$ a.e. is equivalent to proving 
that
\begin{equation}\label{102}
\int_{\C^{2n}}g(z,w)p(z,w)V_t(z,w)dzdw=0
\end{equation} 
for all $p\in\CP(\R^{4n}).$ The assumption $T_g=0$ gives us for all 
$\alpha,\beta,\mu,\nu$
$$\langle T_g \phi_{\alpha,\beta},\phi_{\mu,\nu}\rangle_{\CB_t^*(\C^{2n})} = \int_{\mathbb{C}^{2n}}g(y,v)\phi_{\alpha,\beta}(z,w)\overline{\phi_{\mu,\nu}}(z,w)W_t(z,w)dz dw=0.$$
Again by using the explicit form of $\phi_{\alpha,\beta}$ we get
$$\int_{\mathbb{C}^{2n}}g(y,v)P_{\alpha,\beta}(z,w)\overline{P_{\mu,\nu}}(z,w)V_t(z,w)dzdw=0.$$

We claim that for every $ \alpha, \beta$, the monomial 
$z^{\alpha}w^{\beta}$ belongs to the linear span of 
$ \{ P_{\mu,\nu}(z,w): |\mu|+|\nu|=|\alpha|+|\beta| \}.$ This claim would 
then prove $(\ref{102})$ which in
turn would prove $g=0$ a.e. In fact, once we have the claim $(\ref{102})$ will 
be true for all polynomials of the form $ p(z,w) = z^{\alpha}w^{\beta}
\bar{z}^{\mu}\bar{w}^{\nu} $ which in turn will prove  $(\ref{102})$ for 
all monomials $ x^\alpha y^\beta u^\mu v^\nu $ and hence for all polynomials.
Returning to the claim it is sufficient to prove it  for $z,w $ purely real. 
We know that the special Hermite functions $\Phi_{\mu,\nu}(x,u)$ give all
the eigenfunctions of the dilated Hermite operator $H(1/2)= 
-\triangle+\frac{1}{4}(|x|^2+|u|^2)$
on $\R^{2n}.$ (see \cite{Th5}) More precisely,
$$H(1/2)\Phi_{\mu,\nu}=(|\mu|+|\nu|+n)\Phi_{\mu,\nu}.$$ If 
$ H_{\alpha,\beta}(x,u) $ stand for the (ordinary) Hermite polynomials on 
$ \R^{2n} $ (adapted to $ H(1/2) $) then it can be written as a linear 
combination of  $ P_{\mu,\nu}(x,u)$ with $ |\mu|+|\nu|=|\alpha|+|\beta| .$
It is well known that $x^{\alpha}u^{\beta} $ can be written as a linear 
combination of  $ H_{\mu,\nu} $ and hence as a linear combination of 
$ P_{\mu,\nu} $ as well. Thus $ g $ is orthogonal to all polynomials in 
$ L^2(\C^{2n},V_t) $ and this proves the result.
\end{proof}

We now characterise all the symbols $ g $ for which $ e^{tL}T_g e^{-tL}$ 
reduces to a Weyl multiplier. For this characterisation we need to consider 
symbols $ g $ so that $g_{a,b}(z,w):=g(z+a,w+b)$ 
satisfies condition $(**)$ for all $(a,b)\in\R^{2n}$.
\begin{thm} Let $g_{a,b}$ satisfy $(**)$ for all $ (a,b)\in\R^{2n} $ and 
let the corresponding
$T_g$ be a bounded operator on $\CB_t^*(\C^{2n})$. Then
$ e^{tL}T_g e^{-tL}$ is a right Weyl multiplier if and only if $g(z,w)=g(iy,iv).$
\end{thm}
\begin{proof} Let us first assume that $ T_g $ is bounded and corresponds to
a right Weyl multiplier $ M_t.$  As proved in \cite{D} we know that
$ M_t= W(\sigma)$, for some  $\sigma\in \mathcal{S}'(\mathbb{R}^{2n}).$
Therefore, we have
\begin{equation}\label{f}
e^{tL}T_ge^{-tL}f=f\times\sigma,
\end{equation}
for all  $f \in L^2(\R^{2n}).$ Recall that the twisted translations of 
functions on $\mathbb{R}^{2n}$ are defined by
$$\tau(a,b)f(x,u):=e^{-i/2(a.u-b.x)}f(x-a,u-b), (a,b)\in\mathbb{R}^{2n}.$$
Clearly, $\tau(a,b)$ is a unitary map on $L^2(\mathbb{R}^{2n}).$ It is easy to
check that $\tau(a,b)f\times g= \tau(a,b)(f\times g)$ when
$f,g\in L^2(\mathbb{R}^{2n}).$ This implies that $e^{tL}T_ge^{-tL}$ commutes 
with twisted translations (see (\ref{f})). As $e^{-tL}f=f\times p_t $,
$e^{-tL}f$ is equivariant under twisted translations.
By using the fact that $e^{-tL}$ is a unitary map from
$L^2(\mathbb{R}^{2n})$ onto $\CB_t^*(\C^{2n})$ and its equivariance under
twisted translations, we get that
$$\langle\tau(a,b)F,\tau(a,b)G\rangle_{\CB_t^*(\C^{2n})}=
\langle F,G\rangle_{\CB_t^*(\C^{2n})}$$
for all $(a,b)\in \mathbb{R}^{2n}$ and $F,G\in\CB_t^*(\C^{2n})$. (Here,
$\tau(a,b)$ on $\CB_t^*(\C^{2n})$ is the natural extension to holomorphic
functions.)

We will now show that $g(z,w)=g(iy,iv)$ a.e. In view of Theorem 4.7 and the 
hypothesis on $ g_{a,b} $ it is enough to show that 
$$\langle T_g\phi_{\alpha,\beta},\phi_{\mu,\nu}\rangle_{\CB_t^*(\C^{2n})}
= \langle T_{g_{a,b}}
\phi_{\alpha,\beta},\phi_{\mu,\nu}\rangle_{\CB_t^*(\C^{2n})}$$
for all $(a,b)$ and $\alpha, \beta, \mu, \nu .$ But this can be easily shown 
to be true. Indeed,  by making the change of variable 
$(z,w)\longrightarrow(z+a,w+b)$ in the equation
$$\langle T_g\tau(a,b)\phi_{\alpha,\beta},\tau(a,b)
\phi_{\mu,\nu}\rangle_{\CB_t^*(\C^{2n})} = $$
$$\int_{\C^{2n}}g(z,w)\tau(a,b)\phi_{\alpha,\beta}(z,w)\tau(a,b)\overline{\phi_{\mu,\nu}(z,w)}W_t(z,w)dzdw $$
it is easy to see that
$$ \langle T_g\tau(a,b)\phi_{\alpha,\beta},\tau(a,b)
\phi_{\mu,\nu}\rangle_{\CB_t^*(\C^{2n})}=\langle T_{g_{a,b}}
\phi_{\alpha,\beta},\phi_{\mu,\nu}\rangle_{\CB_t^*(\C^{2n})}.$$ 
Therefore, it is enough to show that
\begin{equation}\label{h}
\langle T_g\phi_{\alpha,\beta},\phi_{\mu,\nu}\rangle_{\CB_t^*(\C^{2n})}
= \langle T_g\tau(a,b)\phi_{\alpha,\beta},\tau(a,b)
\phi_{\mu,\nu}\rangle_{\CB_t^*(\C^{2n})}
\end{equation}
for all $(a,b) \in \R^{2n} $ and multi-indices $(\alpha,\beta) $ and 
$(\mu,\nu).$ In other words, we need to show that $T_g$ commutes with 
twisted translations, which is immediate as  $e^{tL}T_ge^{-tL}$ commutes 
with twisted translations and $e^{-tL}$ is equivariant under them.
This proves the first part of the theorem.

Conversely, assume that $g(z,w)=g(iy,iv)$ and $T_g$ is bounded.
Clearly, 
$$\langle T_g\phi_{\alpha,\beta},\phi_{\mu,\nu}\rangle_{\CB_t^*(\C^{2n})}
= \langle T_{g_{a,b}}
\phi_{\alpha,\beta},\phi_{\mu,\nu}\rangle_{\CB_t^*(\C^{2n})}$$ 
for all $(a,b)\in\R^{2n}.$
As shown earlier, this implies that
$$\langle T_g\phi_{\alpha,\beta},\phi_{\mu,\nu}\rangle_{\CB_t^*(\C^{2n})}
= \langle T_g\tau(a,b)\phi_{\alpha,\beta},\tau(a,b)
\phi_{\mu,\nu}\rangle_{\CB_t^*(\C^{2n})}.$$ So, $T_g$ commutes with
twisted translations. Again, as shown before $e^{tL}T_ge^{-tL}$
commutes with twisted translations as well. This means that there
exists $\sigma\in\CS'(\R^{2n})$ such that $$e^{tL}T_ge^{-tL}f=f\times\sigma$$
for all $f\in L^2(\R^{2n}).$ When we take the Weyl transform on
both sides we get $W(e^{tL}T_ge^{-tL}f)=W(f)W(\sigma).$ As $T_g$ is
bounded, this proves that $e^{tL}T_ge^{-tL}$ is a right Weyl multiplier.
\end{proof}
\section{Toeplitz operators associated to  symmetric spaces}
\setcounter{equation}{0}

Segal-Bargmann spaces associated to Riemannian symmetric spaces have been
studied by Hall \cite {H1}, Stenzel \cite{S} and Kr\"otz et al \cite{KOS}.
The situation of non-compact symmetric spaces is much more complicated
whereas the compact case is well understood as a weighted Bergman spaces. In
both cases we have Gutzmer's formula using which we can study Toeplitz
operators that correspond to Fourier multilpiers on the underlying group. In
this section we study such operators in the case of compact symmetric spaces,
extending some results of Hall \cite{H2}. the case of noncompact Riemannian
symmetric spaces will be taken up elsewhere.

\subsection{Lassalle-Gutzmer formula } Consider a compact symmetric space
$ X = U/K $ where  $ (U,K) $ is a compact symmetric
pair. We may assume that $ K $ is connected
and $ U $ is semisimple. We let $ \u = \k +\p $ stand for the Cartan
decomposition of $ \u $ and let $ \a $ be a Cartan subspace of $ \p.$
Functions  $ f $ on $ X $ can be viewed as  right $ K-$invariant functions
on $ U.$ If $ \pi \in \hat{U} $ then it can
be shown that $ \hat{f}(\pi) = 0 $ unless $ \pi $ is $ K-$spherical, i.e., the
representation space $ V $ of $ \pi $ has a unique $ K-$fixed vector $u.$ It
then follows that $ \hat{f}(\pi)v = (v,u)\hat{f}(\pi)u $ for any $ v \in V $
which means that $ \hat{f}(\pi) $ is of rank one. Let $ \hat{U}_K $ stand for
the equivalence classes of $K-$spherical representations of $ U .$ Then there
is a one to one correspondence between elements of $ \hat{U}_K $ and a
certain discrete subset $ \CP $ of $ \a^* $ called the set of restricted
dominant weights.

For each $ \lambda \in \CP $ let $ (\pi_\lambda, V_\lambda) $ be a spherical
representation of $ U $ of dimension $ d_\lambda.$ Let
$ \{v_j^\lambda, 1 \leq j \leq d_\lambda \} $ be an orthonormal basis for
$ V_\lambda $ with $ v_1^\lambda $ being the unique $ K$-fixed vector.
Then the functions
$$ \varphi_j^\lambda(g) =\langle\pi_\lambda(g)v_1^\lambda,v_j^\lambda\rangle $$
form an orthogonal family of right $ K -$invariant analytic functions on $ U $
and we can consider them as functions of the symmetric space. When
$ x = g.o \in X, $
we simply denote by $ \varphi_j^\lambda(x) $ the function
$ \varphi_j^\lambda(g.o).$ The function $\varphi_1^\lambda(g)$ is  $ K $
biinvariant, called an elementary spherical function. It is usually denoted
by $ \varphi_\lambda.$
The Fourier coefficients of  $ f \in L^2(X), $ are defined by
$$ \hat{f}_j(\lambda) = \int_X f(x)\overline{\varphi_j^\lambda(x)}dm_0(x)$$
and the Fourier series is written as
$$ f(x) = \sum_{\lambda \in \CP} d_\lambda \sum_{j=1}^{d_\lambda}
\left( \hat{f}_j(\lambda)\varphi_j^\lambda(x)\right).$$ Then the Plancherel theorem reads as
$$ \int_X |f(x)|^2 dm_0(x) = \sum_{\lambda \in \CP} d_\lambda
\left(\sum_{j=1}^{d_\lambda} |\hat{f}_j(\lambda) |^2 \right) .$$

Let $ U_\C $ (resp. $K_\C $) be the universal complexification of $ U $ (resp.
$ K$). The group $ K_\C $ sits inside $ U_\C $ as  a
closed subgroup. We may then consider the complex homogeneous space
$ X_\C = U_\C/K_\C, $ which is a complex variety and gives the
complexification of the
symmetric space $ X = U/K.$ The Lie algebra $ \u_\C $ of $ U_\C $ is the
complexified Lie algebra $ \u_\C = \u +i\u.$ For every $ g \in U_\C $
there exists $ u \in U $ and $ X \in \u $ such that $ g = u \exp iX.$
Let $ \Omega $ be any $ U $ invariant domain in $ X_\C $ and let $ \CO(\Omega)$
stand for the space of holomorphic functions on $ \Omega.$ The group $ U $
acts on $ \CO(\Omega)$ by $ T(g)f(z) = f(g^{-1}z).$ For each
$ \lambda \in \CP $ the matrix coefficients $ \varphi_j^\lambda $ extend
to $ X_\C $ as
holomorphic functions. When $ f \in \CO(\Omega), $ it can be shown that
the series
$$ f(z) = \sum_{\lambda \in \CP} d_\lambda \sum_{j=1}^{d_\lambda}
\hat{f}_j(\lambda)\varphi_j^\lambda(z) $$
converges uniformly over compact subsets of $ \Omega.$ The above series
is called the Laurent expansion of $ f $ and we  have the following formula
known as  Gutzmer's formula for $ X.$

\begin{thm} For every $ f \in \CO(X_\C) $ and
$ H \in i\a, $ we have
$$ \int_U |f(g.\exp(H).o)|^2 dg = \sum_{\lambda \in \CP} d_\lambda
\left(\sum_{j=1}^{d_\lambda} |\hat{f}_j(\lambda) |^2 \right)
\varphi_\lambda(\exp(2H).o).$$
\end{thm}

This theorem is due to Lassalle, see \cite{La1} and \cite{La2} for a proof.
This formula has
been used by Faraut \cite{Fa1} to give an elegant proof of a theorem of
Stenzel \cite{S} on
the Segal-Bargmann transform for the compact symmetric space $ X.$ The second
author has used the same to study holomorphic Sobolev spaces in \cite{Th7}.

\subsection{Segal-Bargmann transform on $ X $} Let $ \Delta $ stand for the
Laplace-Beltrami operator on $ X $ suitably shifted so that its spectrum
consists of $ |\lambda+\rho|^2 $ where $ \lambda \in \CP $ and $ \rho $ is
the half sum of positive roots (see Faraut \cite{Fa1}). The solution of 
the heat
equation associated to $ \Delta $ with initial condition $ f \in L^2(X) $
is given by the expansion
$$ u(x,t) = \sum_{\lambda \in \CP} d_\lambda e^{-t|\lambda+\rho|^2}\left(
\sum_{j=1}^{d_\lambda} \hat{f}_j(\lambda)\varphi_j^\lambda(x)\right) .$$
By defining the heat kernel $ \gamma_t $ by
$$ \gamma_t(x) = \sum_{\lambda \in \CP} d_\lambda e^{-t|\lambda+\rho|^2}
\left(\sum_{j=1}^{d_\lambda} \varphi_j^\lambda(x)\right) $$
the solution can be written as $ u(g,t) = f*\gamma_t(g) $ where the
convolution is taken on $ U.$ For $ f \in L^2(X) $ it can be shown that the
solution $ u $ extends to $ X_\C $ as a holomorphic function. The map taking
$ f $ into $ u(z,t) = f*\gamma_t(g), z = g.o $ is called the Segal-Bargmann
transform and has been studied by Hall \cite{H1}, Stenzel \cite{S} and others.

The image of $ L^2(X) $ under the Segal-Bargmann transform has been
characterised by Stenzel \cite{S} as a weighted Bergman space. The weight function
$ w_t $ is given in terms of the heat kernel on the noncompact dual $ Y $ of
$ X.$ Consider the group $ G = K\exp(i\p) $ whose Lie algebra is $ \k+i\p.$
Under the assumption that $ U $ is semisimple, $ G $ turns out to be a real
semisimple group and $ K $ a maximal compact subgroup. The noncompact dual
is then defined as $ Y = G/K.$ Let $ \Delta_G $ be the Laplace-Beltrami
operator on $ Y $ with heat kernel defined by
$$ \gamma_t^1(g) = \int_{(i\a)^*} e^{-t(|\lambda|^2+|\rho|^2)}\psi_\lambda(g)
|c(\lambda)|^{-2} d\lambda.$$ Here $ \psi_\lambda $ are the spherical
functions on $ Y.$ Define a weight function $ w_t(z) $ on
$ X_\C $ by $ w_t(z) = \gamma_{2t}^1(\exp(2H)), z = u\exp(H),
u \in U, H \in i\a.$ Then we have the following result.

\begin{thm} The Segal-Bargmann transform is an isometric isomorphism between
$ L^2(X) $ and the space of all holomorphic functions on $ X_\C $ that are
square integrable with respect to $ w_t(z) dm $ where $ dm $ is the invariant
measure on $ X_\C.$
\end{thm}

This theorem is due to Stenzel \cite{S}; for an elegant proof using Gutzmer's
formula see Faraut \cite{Fa1}. The key ingredient in Faraut's proof is Lassalle's
formula and the following relation between $ \varphi_\lambda $ and
$ \psi_\lambda ,$ namely
$$ \varphi_\lambda(\exp(H)) = \psi_{-i(\lambda+\rho)}(\exp(H)), H \in i\a.$$
To conclude this subsection let us recall the following integration formulas
on $ X_\C $ and $ Y $:
$$ \int_{X_\C} f(x) dm(x) = c \int_U \int_{i\a} f(u\exp(H).o)J(H) du dH ,$$
$$ \int_{Y} f(x) dm_1(x) = c \int_K \int_{i\a} f(u\exp(H).o)J_1(H) du dH .$$
Here the Jacobians $ J $ and $ J_1 $ are defined in terms of the roots, see
Faraut \cite{Fa1}. We need the following fact that $ J(H) = J_1(2H).$

\subsection{Toeplitz operators and Fourier multipliers} Given a symbol
$ g(z) $ defined on $ X_\C $ we consider the Toeplitz operator $ T_g $ on the
Segal-Bargmann space $ HL^2(X_\C,w_t).$ In this subsection we are interested
in finding symbols $ g $ so that $ e^{t\Delta}T_ge^{-t\Delta} $ is a Fourier
multiplier on $ L^2(X, dm_0).$ Given a bounded function $ a(\lambda) $ the
Fourier multiplier $ a(D) $ is defined by
$$ a(D)f(x) =  \sum_{\lambda \in \CP} d_\lambda a(\lambda)
\left( \sum_{j=1}^{d_\lambda}\hat{f}_j(\lambda)\varphi_j^\lambda(x)\right) $$
for all $ f \in L^2(X,dm_0).$ It is clear that $ a(D) $ is bounded if
and only if $ a $ is bounded. Using Gutzmer's formula we can easily prove the
following result.

\begin{thm} Suppose $ h $ is a $K-$biinvariant distribution on $ G $ so that
$ h*\gamma_{2t}^1 $ is well defined. Let
$ g(z)w_t(z) = h*\gamma_{2t}^1(\exp(2H)) $
whenever $ z = u\exp(H), u \in U, H \in i\a.$ Then $  e^{t\Delta}T_g
e^{-t\Delta}= a(D) $ where
$$ a(\lambda) = \int_{i\a} h(\exp(H))\psi_{-i(\lambda+\rho)}
(\exp(H)) J_1(H) dH.$$
\end{thm}
\begin{proof} When $ F, F' \in  HL^2(X_\C,w_t) $ we can use the polarised form
of Gutzmer's formula to get
$$ \int_U F(u\exp(H).o)\overline{F'(u\exp(H).o)} du $$
$$ =  \sum_{\lambda \in \CP} d_\lambda e^{-2(|\lambda+\rho|^2)t}
\left(\sum_{j=1}^{d_\lambda}\hat{f}_j(\lambda) \overline{\hat{f'}_j(\lambda)}
\right)\varphi_\lambda(\exp(2H)) $$
where $ F = f*\gamma_t $ and $ F' = f'*\gamma_t.$ Integrating the above with
respect to $ h*\gamma_{2t}^1(\exp(2H)) J(H) dH $ and recalling the
definition of $g(z) $ we obtain
$$ \int_{X_\C} F(z)\overline{F'(z)}g(z)w_t(z)dm(z)  =  \sum_{\lambda \in \CP}
d_\lambda e^{-2(|\lambda+\rho|^2)t}$$
$$ \times
\left(\sum_{j=1}^{d_\lambda}\hat{f}_j(\lambda) \overline{\hat{f'}_j(\lambda)}
\right) \int_{i\a} \varphi_\lambda(\exp(2H))
h*\gamma_{2t}^1(\exp(2H)) J(H) dH  .$$ As $ J(H) = J_1(2H) $ and
$ \varphi_\lambda(\exp(H)) = \psi_{-i(\lambda+\rho)}(\exp(H)) $ the integral
on the right hand side reduces to
$$ \int_{i\a}h*\gamma_{2t}^1(\exp(H))\psi_{-i(\lambda+\rho)}(\exp(H))J_1(H)
dH = e^{2(|\lambda+\rho|^2)t}\tilde{h}(-i(\lambda+\rho)).$$ Thus we have
$$ \int_{X_\C} T_gF(z)\overline{F'(z)}w_t(z)dm(z) = \int_X a(D)f(x)
\overline{f'(x)} dm_0(x) $$
which proves the theorem.
\end{proof}

\begin{rem} When we take $ h $ to be the distribution  $ p(\Delta)\delta_e $
where $ p $ is a polynomial it follows that $ a(\lambda) =
p(-|\lambda+\rho|^2) $ so that $ a(D) = p(i\Delta).$ Hence the differential
operator $ p(i\Delta) $ corresponds to the Toeplitz operator with $ T_g $
with symbol $ g(z) = \gamma_{2t}(\exp(2H))^{-1}
p(\Delta)\gamma_{2t}(\exp(2H)), z = u\exp(H).o .$ In the context of compact
Lie groups $ U $, Hall \cite{H2} has considered more general differential operators
on $ U $  and studied the symbols of Toeplitz operators corresponding to them
using a different method.
\end{rem}

\subsection{Some remarks on compact Lie groups} Let us rewrite our theorem in
the previous section as follows. Given a $ K-$biinvariant function $ g_0 $ on
$ Y = G/K $ define $ g(z) = g_0(\exp(2H)), z = u\exp(H).o, H \in i\a.$ Then we 
have

\begin{cor} Let $ g $ be as above. Then the Toeplitz operator $ T_g $ is 
bounded on $ HL^2(X_\C,w_t dm) $ if and only if
$$  | \int_{i\a} g_0(\exp(H))\gamma_{2t}^1(\exp(H))
\psi_{-i(\lambda+\rho)}(\exp(H) J_1(H) dH | \leq C e^{2t|\lambda+\rho|^2} $$
for all $ \lambda \in \CP.$
\end{cor}

Let $ U $ be a compact semisimple Lie group which can be treated as a compact
symmetric space. In this case the group $ G $ turns out to be a complex Lie
group and hence the heat kernel $ \gamma_t^1 $ is explicitly known, see
Gangolli \cite{G}. We also have explicit expressions for the spherical functions
$ \varphi_\lambda $ (Weyl character formula) and $ \psi_\lambda.$ More
precisely,
$$ \psi_\lambda(\exp(H)) = \frac{ \sum_{s \in W}
c(s\lambda) e^{i s\lambda(H)}}{\Pi_{\alpha \in Q}(e^{\alpha(H)}-
e^{-\alpha(H)})} $$
where $ W $ is the Weyl group, $ c $ is the Harish-Chandra $c-$function and
$ Q $ is the set of positive roots. The heat kernel is given by
$$ \gamma_t^1(\exp(H)) = C_t e^{-t|\rho|^2}\Pi_{\alpha \in Q} \frac{\alpha(H)}
{(e^{\alpha(H)}-e^{-\alpha(H)})} e^{-\frac{1}{4t}|H|^2}.$$
These two results are proved in Gangolli \cite{G}; see also Helgason
\cite{H}. Defining
$ \pi(\lambda) = \Pi_{\alpha \in Q}\alpha(H_\lambda) $ where $ H_\lambda \in
i\a $ corresponds to $ \lambda $ we have the simple formula
$  c(\lambda) = \pi(\rho)/\pi(i\lambda).$

The Jacobian factor $ J_1(H) $ appearing in the integration formula for $ Y =
G/K $ is also expressible in terms of the roots $ \alpha \in Q.$ Thus it can
be checked that
$$ \int_{i\a} g_0(\exp(H))\gamma_{2t}^1(\exp(H))
\psi_{-i(\lambda+\rho)}(\exp(H) J_1(H) dH $$
$$ = C_t e^{-t|\rho|^2} \sum_{s \in W} c(-is(\lambda +\rho))
\int_{i\a} g_0(\exp(H))
 e^{s(\lambda + \rho)(H)}\pi(H) e^{-\frac{1}{8t}|H|^2} dH.$$
Note that when $ g_0 = 1 $ the integral
$$ \int_{i\a} \gamma_{2t}^1(\exp(H))
\psi_{\lambda}(\exp(H)) J_1(H) dH $$
reduces to
$$ e^{-2t|\rho|^2} \sum_{s \in W} c(s\lambda)
\int_{i\a} \pi(H) e^{is\lambda(H)} e^{-\frac{1}{8t}|H|^2} dH $$
$$ = e^{-2t|\rho|^2} \left(\sum_{s \in W} c(s\lambda)\pi(is\lambda)\right)
e^{-2t|\lambda|^2}= C_t
e^{-2t(|\lambda|^2 +|\rho|^2)} $$ which is the defining relation for the heat
kernel.

\begin{thm} Let $ T_g $ be a Toeplitz operator on the Segal-Bargmann space
associated to a compact Lie group $ U $ where $ g(u\exp(H).o) = g_0(\exp(H)),
u \in U, H \in i\a.$ Then $ T_g $ is bounded if and only if
$$  |\int_{i\a} g_0(\exp(H))e^{(\lambda + \rho)(H)}
\pi(H) e^{-\frac{1}{8t}|H|^2} dH | \leq C_t |\pi(i(\lambda +\rho))|
e^{2t|\lambda+\rho|^2}.$$
\end{thm}

Defining $ g_1(H) = g_0(\exp(H))\pi(H) $ the above condition can be put in
the form
$$ |\int_{i\a} g_1(H)e^{-\frac{1}{8t}|H-4t(\lambda+\rho)|^2} dH |
\leq C_t  |\pi(i(\lambda +\rho))| $$
for all $ \lambda \in \CP.$ This has an obvious resemblance with the
sufficient condition we obtained for the Fock spaces.

\begin{center}
{\bf Acknowledgments}
\end{center}
The second author is supported in part by J. C. Bose Fellowship from 
the Department of Science and Technology (DST).

\end{document}